\documentclass[a4paper]{article}
\usepackage{amsmath}
\usepackage{lmodern}
\usepackage{amsmath}
\usepackage{amssymb}
\usepackage{mathrsfs}
\usepackage{graphicx}

\newtheorem{theorem}{Theorem}

\newtheorem{corollary}[theorem]{Corollary}

\newtheorem{definition}[theorem]{Definition}
\newtheorem{example}[theorem]{Example}

\newtheorem{lemma}[theorem]{Lemma}

\newtheorem{remark}[theorem]{Remark}

\newenvironment{proof}[1][Proof]{\noindent\textbf{#1.} }{\ \rule{0.5em}{0.5em}}
\newcommand{\R}{{\mathbb R}}

\newcommand*{\qed}{\hfill\ensuremath{\blacksquare}}%

\begin{document}
\title{Existence and uniqueness of positive solutions for nonlinear fractional mixed problems.}
\date{}
\author{Alberto Cabada$^1$ and Wanassi Om Kalthoum$^2$\\
	$^1$Departamento de Estat\'istica, An\'alise Matem\'atica e Optimizaci\'on\\
	Instituto de Matem\'aticas, Facultade de Matem\'aticas, \\
	Universidade de Santiago de Compostela, Spain.\\
	alberto.cabada@usc.es\\
$^2$Department of Mathematics\\University of Monastir, Tunisia.\\kalthoum.wannassi@gmail.com}
\maketitle
\begin{abstract}
This paper is devoted to study the existence and uniqueness of solutions of a  one parameter family of nonlinear fractional differential equation with mixed boundary value conditions. 
Riemann-Liouville fractional derivative is considered. An exhaustive study of the sign of the related Green's function is done.

Under suitable assumptions on the asymptotic behavior of the nonlinear  part of the equation at zero and at infinity, and by application of the fixed theory of compact operators defined in suitable cones, it is proved the existence of at least one solution of the considered problem. Moreover it is developed the method of lower and upper solutions and it is deduced the existence of solutions by a combination of both techniques. In some particular situations, the Banach contraction principle is used to ensure the uniqueness of solutions.
\end{abstract}


\section{Introduction}

Fractional calculus is a very well known tool in the study of both pure and theoretical mathematical
discipline. In last decades, it has had a significant growth since this discipline has been gained presence because of its practical applications. The main difference of this kind of calculus is that it takes into account the values of the considered functions on previous instants to the one where it is studied. In particular, to evaluate the fractional derivative of a function on a given value $t$, it is necessary to know its definition on the whole  interval $[0,t)$. This property make the equations with this kind of derivatives specially suitable for problems ``with memory'' (see \cite{M.} and references therein) in which, on the contrary to the Ordinary Differential Equations, the response of the system is not immediate. So, in current research, fractional differential equations have arisen in mathematical models of
systems and processes in various fields such as, among others, aerodynamics, acoustics, robotics, electromagnetism, signal processing, mechanics,  control theory,  population dynamics or finance. Classical and recent results may be found on the monographs \cite{butzer,gorenflo,ref1,mainardi,muler,oldham,podlubny,ref8} and references therein.

In this article, we discuss the existence of solutions of the following nonlinear fractional differential equation with mixed boundary conditions
\begin{eqnarray}
\begin{cases}
D^{\alpha} u(t)-\lambda u(t)+ f(t,t^{2-\alpha}u(t))=0 , \quad t \in I:= [0,1],\\
\displaystyle \lim_{t\rightarrow 0^+} t^{2-\alpha}u(t)=u'(1)=0,
\end{cases} \label{pr}
\end{eqnarray}
where $1<\alpha \leq 2$, $\lambda \in \mathbb{R}$, $D^{\alpha}$ is the Riemann-Liouville fractional derivative and $f$ is a continuous function. 

We look for solutions $u:I \to \R$ such that function $t^{2-\alpha}\, u(t) \in C^1(I)$. Notice that, as a direct consequence, we deduce that, in particular, $u \in C^1((0,1])$. Moreover, it may be discontinuous at $t=0$.

Many results in this direction have been obtained in the literature for second order Ordinary Differential Equations. We  point out that for the, so-called, Hill's equation, $u''(t)+a(t)\, u(t)+ f(t,u(t))=0$, there is a huge bibliography where the existence of solutions is obtained. Many of them are on the basis  of the oscillation properties related to the linear part ot the equation. One may see, for instance, the monographs \cite{cacidsomoza,coppel, magnus} and the recent publications \cite{cacid1, torres1, zhang, zhanli}. In all of them the oscillation properties of the possible solutions are fundamental to deduce the existence of positive solutions of the considered problems.

In our case, due to the definition of the Riemann-Liouville fractional derivative and the lack of regularity of the solutions we are looking for, the arguments used on those references are not suitable for Problem \eqref{pr}. So, we will center our efforts in the construction of the Green's function and in to define operators in suitable spaces where to save the lack of the regularity of the obtained integral kernel.

Related problems are considered in several papers by means of the construction of the integral kernel of the considered operator. For instance, in \cite{ghanmi} the following problem is considered:
\begin{eqnarray*}
\begin{cases}
D^{\alpha} u(t)= f(t,u(t))=0 , \quad t \in I,\\
u(0)=\beta\,u(1)-\gamma\,u(\eta)=0,
\end{cases} 
\end{eqnarray*}
for $1<\alpha \leq 2$, $\beta$, $\gamma$, $\eta >0$ such that $\beta-2\, \gamma\,\eta^{\alpha-1}>0$.  On the paper, it is proved the existence of solution for a nonnegative and bounded function $f$ that satisfies some Lipschitz condition.

In \cite{ali}, it is considered the problem
\begin{eqnarray*}
\begin{cases}
D^{\alpha} u(t)+ f(t,u(t))=0 , \quad t \in I,\\
u(0)=u'(0)=u''(0)=u''(1)=0.
\end{cases} 
\end{eqnarray*}
In this case for $\alpha \in (3,4]$. There, the authors proved the validity of the monotone iterative techniques and  deduce some kind of stability for the obtained solutions.

In \cite{benchehen} it is studied the problem (with $\alpha \in (2,3]$)
\begin{eqnarray*}
	\begin{cases}
		D^{\alpha} u(t)+ f(t,u(t))=0 , \quad t \in I,\\
		u(0)=u'(0)=u'(1)=0.
	\end{cases} 
\end{eqnarray*}
The existence of solutions is deduced, for non negative Carath\'edory functions via degree theory.

On \cite{ref9} the Positiveness of the Green's function related to the linear part of the Dirichlet boundary problem

\begin{eqnarray*}
\begin{cases}
-D^{\alpha} u(t)+\lambda u(t)= f(t,u(t)) , \quad t \in I,\\
u(0)=u(1)=0,
\end{cases} 
\end{eqnarray*}
with $1<\alpha \leq 2$, $\lambda >0$, $D^{\alpha}$ and, once again, being the continuity of $u$  at $t=0$  allowed.

It is important to point out that in all the previous references the regularity imposed to the possible solutions imply their continuity at $t=0$.

Our approach is on the line of reference \cite{ref2}, where the following Dirichlet boundary problem is considered: 

\begin{eqnarray}
\begin{cases}
D^{\alpha} u(t)-\lambda u(t)+ f(t,t^{2-\alpha}u(t))=0 , \quad t \in I:= [0,1],\\
\displaystyle \lim_{t\rightarrow 0^+} t^{2-\alpha}u(t)=u(1)=0,
\end{cases} \label{p-Dir}
\end{eqnarray}
with $1<\alpha \leq 2$, $\lambda \in \mathbb{R}$, $f$ a continuous function and $t^{2-\alpha}\, u(t) \in C^1(I)$.

In this paper we will make an spectral analysis for  equation \eqref{pr} and give a comparison with the spectrum of Problem \eqref{p-Dir}, obtained in \cite{ref2}. So we need to combine the results for both problems in order to ensure the constant sign of the obtained Green's function. So, under suitable assumptions on the asymptotic behavior of the non negative nonlinear function $f(t,x)$ at $t=0$ and $t=+\infty$, and by application of the fixed theory of compact operators defined in suitable cones, it is proved the existence of at least one solution of the considered problem. Moreover, when the nonlinear is not necessarily of a constant sign, it is developed the method of lower and upper solutions and it is deduced the existence of solutions by a combination of both techniques. In some particular situations, the Banach contraction principle is used to ensure the uniqueness of the solutions of the considered problem.

The main tool used consists on the construction of the Green's function related to the linear problem
\begin{eqnarray}
\begin{cases}
D^{\alpha} u(t)-\lambda u(t)+ y(t)=0 , \quad t \in I,\\
\displaystyle \lim_{t\rightarrow 0^+} t^{2-\alpha}u(t)=u'(1)=0.
\end{cases} \label{pr1}
\end{eqnarray}

Once we have such expression, it is obtained the exact interval of the parameter $\lambda$ for which such function is positive on its square of definition. To this end, we make a spectral analysis of the linear operator in a suitable space.

The paper is scheduled as follows: In Section 2 are introduced some preliminary results that will allow us to obtain, in Section 3, the Green's function related to problem \eqref{pr1}. Next two sections are devoted to deduce the existence of solutions of the nonlinear problem \eqref{pr}. Such results are deduced from Fixed Point Theorems in cones. Section 6 is devoted to prove the validity of the method of lower and upper solutions and in last section some examples are given.

\section{Preliminaries}
In this section, we present some necessary definitions from fractional calculus. And we give some theorems that will be used to prove our results in next sections.
\begin{definition}\cite{ref1}
\label{def1}
\textrm{The Riemann-Liouville fractional integral of order }$%
\alpha >0\ $\textrm{for a measurable function }$f:(0,+\infty )\rightarrow 
\mathbb{R}
$\textrm{\ is defined as }%
\begin{equation*}
I^{\alpha }f(t)=\frac{1}{\Gamma (\alpha )}\int_{0}^{t}(t-s)^{\alpha
-1}f(s)ds,\ t>0,
\end{equation*}%
\textrm{where }$\Gamma $\textrm{\ is the Euler Gamma function, provided that
the right-hand side is pointwise defined on }$(0,+\infty ).$
\end{definition}
\begin{definition}\cite{ref1}
\label{def2} \textrm{The Riemann-Liouville fractional derivative of order }$%
\alpha >0\ $\textrm{for a measurable function }$f:(0,+\infty )\rightarrow 
\mathbb{R}
$\textrm{\ is defined as }%
\begin{equation*}
D^{\alpha }f(t)=\dfrac{1}{\Gamma (n-\alpha )}\frac{d^n}{dt^n}\int_{0}^{t}(t-s)^{n-\alpha -1}f(s)ds=\frac{d^n}{dt^n}I^{n-\alpha
}f(t),\ 
\end{equation*}%
provided that the right-hand side is pointwise defined on $(0,+\infty )$. \textrm{Here }$n=[\alpha ]+1$\textrm{, where }$[\alpha ]$\textrm{\  denotes the integer part of the real number }$\alpha $.
\end{definition}
\begin{definition} \label{d-MiLe} \cite[p.\ 42]{ref1} 
A two parameter function of the Mittag$-$Leffler $E_{\alpha,\beta}(x)$ is defined by the series expansion 
$$  E_{\alpha,\beta}(x)=\sum_{k=0}^{\infty} \frac{x^k}{\Gamma(\alpha k+\beta)}, \ \ \alpha,\beta >0, \ x \in \mathbb{R} .$$
For $\beta=1$,  $E_{\alpha,1}$ coincides with the usual Mittag$-$Leffler function $E_{\alpha}$.   
\end{definition}
\begin{theorem}\cite[ Theorem 5.1, p.\ 284]{ref1}\label{th1} 
Let $n-1< \alpha \leq n$ ($n \in \mathbb{N}$) and $\lambda \in \mathbb{R}$ be given. Then the functions 
$$ u_{j}(t)= t^{\alpha-j} E_{\alpha,\alpha+1-j}(\lambda t^\alpha), \ \ \ j=1,\ldots, n,$$
yield a fundamental system of solutions of the equation
$$ D^\alpha u(t)-\lambda u(t)=0, \ \ t>0.$$
\end{theorem}
\begin{theorem} \cite[Theorem 5.7, p.\ 302]{ref1}\label{th2}
Let $n-1< \alpha \leq n$ ($n \in \mathbb{N}$) and $\lambda \in \mathbb{R}$, and let $f$ be a given real function defined on $ \mathbb{R}$. Then the equation 
$$ D^\alpha u(t)-\lambda u(t)=f(t), \ \ t>0,$$
is solvable and its general is given by
$$u(t)=  \int_{0}^{t} (t-s)^{\alpha-1} E_{\alpha,\alpha}[\lambda (t-s)^\alpha] f(s)ds + \sum_{j=1}^{n}c_{j}t^{\alpha-j} E_{\alpha,\alpha+1-j}(\lambda t^\alpha), $$
with $c_{j} \in \mathbb{R}$, $j=1,....,n$, arbitrarily chosen.
\end{theorem}
Let $C(I)$ the Banach space of all continuous functions defined on $I$ endowed with the norm $\Vert f \Vert=:\max\lbrace \vert f(t)\vert: t\in I \rbrace$.\\
Define for $t \in I$, $f_{\gamma}(t)=t^\gamma f(t).$ 
Let $C_{\gamma}(I)$, $\gamma \geq 0$ be the space of all functions $f$ such that $f_{\gamma} \in C(I)$. 
It is well known  that $C_{\gamma}(I)$ is a Banach space endowed with the norm 
$$\Vert f \Vert_{\gamma}=:\max\lbrace t^{\gamma} \vert f(t)\vert: t\in I \rbrace.$$
\section{Green's function}
In this section we obtain the explicit expression of the Green's function associated to the linear problem (\ref{pr1}).

First of all, we must determine the eigenvalues of the homogeneous problem (\ref{pr1}) (when $y\equiv 0$ on $I$). So, by Theorem \ref{th1}, its general solution is given by 
\begin{equation}\label{eq1}
u(t)=C_{1}  t^{\alpha-1} E_{\alpha,\alpha}(\lambda t^\alpha) +C_{2} t^{\alpha-2}E_{\alpha,\alpha-1}(\lambda t^\alpha),
\end{equation}
with $C_1$, $C_2 \in \R$.

Thus,
\begin{equation*}
t^{2-\alpha}u(t)=C_{1} t E_{\alpha,\alpha}(\lambda t^\alpha)+C_{2} E_{\alpha,\alpha-1}(\lambda t^\alpha).
\end{equation*}

Since $0=\lim\limits_{t\rightarrow 0^+} t^{2-\alpha}u(t)$, we get 
$$0=C_{2} E_{\alpha,\alpha-1}(0)=\frac{C_{2}}{\Gamma(\alpha-1)},$$
and so, $C_{2}=0$.

Differentiating (\ref{eq1}), we obtain, for $t>0$,
$$ u'(t)=C_{1} t^{\alpha-2} E_{\alpha,\alpha-1}(\lambda t^\alpha).$$ 

Then $u'(1)=0$ implies that 
$$C_{1} E_{\alpha,\alpha-1}(\lambda)=0.$$

Therefore, $\lambda$ is an eigenvalue of problem (\ref{pr1}) if and only if
\begin{equation}\label{mit}
 E_{\alpha,\alpha-1}(\lambda)=0. 
\end{equation}

It is clear, from Definition \ref{d-MiLe}, that all the zeros of previous equation must be negative. This equation will have for any $\alpha \in (1,2]$ a finite number of negative zeros. Numerically, in Table \ref{t-1}, are compiled the estimations for the first negative zero, which is denoted by $\lambda_1^*$.

	\begin{table}[]
	\centering
	\caption{First eigenvalue $\lambda^*_1$ of Problem (\ref{pr1})\label{t-1}}
\hspace*{-3cm}	\begin{tabular}{l|llllllllll}
		$\alpha$ & 1.1 & 1.2 & 1.3 & 1.4 & 1.5 & 1.6 & 1.7 & 1.8 & 1.9 & 2\\
		\hline
		\hline
		$ \lambda^*_{1}$ & -0.104812 & -0.221832& -0.355588 &-0.511676 &-0.697078 &-0.920556 &-1.19319&-1.52904&-1.9461& -2.4674
	\end{tabular}
\end{table}

In  next result, we deduce the expression of the Green's function associated to the linear problem (\ref{pr1}).
\begin{theorem}\label{th3}
Let $y \in C(I)$, $1<\alpha\leq 2$ and $\lambda \in \mathbb{R}$ be such that $E_{\alpha,\alpha-1}(\lambda)\neq 0$. Then problem (\ref{pr1}) has a unique solution 
$$ u \in C^1_{2-\alpha}(I)=\{u:I\to \R; t^{2-\alpha}\, u(t) \in C^1(I)\},$$
 given by 
$$ u(t)=\int_{0}^{1} G(t,s)y(s)ds,$$
where 
\begin{eqnarray}\label{G}
G(t,s)=
\begin{cases}
\frac{ t^{\alpha-1} E_{\alpha,\alpha}(\lambda t^\alpha) E_{\alpha,\alpha-1}(\lambda(1-s)^\alpha)}{(1-s)^{2-\alpha}E_{\alpha,\alpha-1}(\lambda)}- (t-s)^{\alpha-1} E_{\alpha,\alpha}(\lambda (t-s)^\alpha), \ \ \ 0\leq s \leq t \leq 1, \\
\frac{ t^{\alpha-1} E_{\alpha,\alpha}(\lambda t^\alpha) E_{\alpha,\alpha-1}(\lambda(1-s)^\alpha)}{(1-s)^{2-\alpha}E_{\alpha,\alpha-1}(\lambda)},\qquad \qquad \qquad \qquad \qquad \qquad \quad \; 0\leq t < s <1.\\
\end{cases}
\end{eqnarray}
\end{theorem} 

Before proving previous result, by simple calculations, we obtain the following result.
\begin{lemma}\label{lm1}
	Let $1<\alpha \leq 2$ and $\lambda \in \mathbb{R}$. Then function $G$ defined by (\ref{G}) satisfies the following properties.
	\begin{itemize}
		\item[(i)] $G$ is a continuous function on $I\times [0,1)$.
		\item[(ii)] $G(0,s)=G(t,0)=\frac{\partial G(t,s)}{\partial t}\vert_{t=1} =0$, for all $t \in I$, and $s \in[0,1)$.
		\item[(iii)] $\lim\limits_{s\rightarrow 1^{-}} \vert G(t,s) \vert =+\infty$, for all $t \in (0,1]$.
		\item[(iv)] $ \int_{0}^{1} \vert G(t,s) \vert ds < \infty$, for all $ t \in I$.
		\item[(v)] $ \int_{0}^{1} \vert \frac{\partial}{\partial t} t^{2-\alpha}\,G(t,s) \vert ds < \infty$, for all $ t \in I$.
	\end{itemize}
\end{lemma}

\noindent{\bf Proof of Theorem \ref{th3}}
Using Theorem \ref{th2}, the solutions of problem (\ref{pr1}) are given by
\begin{equation}\label{eq03}
u(t)=C_{1}  t^{\alpha-1} E_{\alpha,\alpha}(\lambda t^\alpha) +C_{2} t^{\alpha-2}E_{\alpha,\alpha-1}(\lambda t^\alpha)-\int_{0}^{t} (t-s)^{\alpha-1} E_{\alpha,\alpha}(\lambda (t-s)^\alpha) y(s) ds.
\end{equation}

Since $\lim\limits_{t\rightarrow 0^+} t^{2-\alpha}u(t)=0$, it is clear that $C_{2}=0$.

Now, we take derivative of (\ref{eq03}) for $t >0$. So, we have
$$ u'(t)= C_{1} t^{\alpha-2}E_{\alpha,\alpha-1}(\lambda t^\alpha)-\int_{0}^{t} (t-s)^{\alpha-2} E_{\alpha,\alpha-1}(\lambda (t-s)^\alpha ) y(s) ds.$$

Then, condition $u'(1)=0$ implies that $$C_{1}=\frac{1}{E_{\alpha,\alpha-1}(\lambda)}\int_{0}^{1} (1-s)^{\alpha-2} E_{\alpha,\alpha-1}(\lambda (1-s)^\alpha ) y(s)ds.$$

As a consequence, the unique solution $u$ of problem \eqref{pr1} is given by
\begin{eqnarray*}
u(t)&=&\frac{1}{E_{\alpha,\alpha-1}(\lambda)} \int_{0}^{1} (1-s)^{\alpha-2}  t^{\alpha-1} E_{\alpha,\alpha}(\lambda t^\alpha) E_{\alpha,\alpha-1}(\lambda (1-s)^\alpha ) y(s) ds\\
&&- \int_{0}^{t}  (t-s)^{\alpha-1} E_{\alpha,\alpha}(\lambda (t-s)^\alpha) y(s) ds.
\end{eqnarray*}

Finally, we deduce that 
\begin{eqnarray*}
u(t)&= &  \int_{0}^{t}\left( \frac{(1-s)^{\alpha-2}  t^{\alpha-1} E_{\alpha,\alpha}(\lambda t^\alpha) E_{\alpha,\alpha-1}(\lambda (1-s)^\alpha )}{E_{\alpha,\alpha-1}(\lambda)}- (t-s)^{\alpha-1} E_{\alpha,\alpha}(\lambda (t-s)^\alpha) \right) y(s) ds\\
&+ & \int_{t}^{1} \frac{(1-s)^{\alpha-2}  t^{\alpha-1} E_{\alpha,\alpha}(\lambda t^\alpha) E_{\alpha,\alpha-1}(\lambda (1-s)^\alpha )}{E_{\alpha,\alpha-1}(\lambda)} y(s) ds\\
&= & \int_{0}^{1} G(t,s)y(s)ds.
\end{eqnarray*}

From Theorem \ref{th3} $(v)$, the proof is completed.
\qed \\

The following lemma describes the set of real parameters $\lambda$ for which the Green's function has a constant sign. To this end, we introduce $\lambda_{1}^{*}$ as the biggest negative zero of $E_{\alpha,\alpha-1}(\lambda)$.  
\begin{lemma}\label{lm2}
 Let $G$ be the Green's function associated to problem (\ref{pr1}) and $\lambda_{1}^{*}$ be the first negative zero of $E_{\alpha,\alpha-1}(\lambda)=0$. Then for $1< \alpha \leq 2$, it is satisfied that
 $$ G(t,s) >0 \ \ for \ all \ t, s \in(0,1)\ if \ and \ only \ if \ \ \lambda> \lambda_{1}^{*}.$$ 
\end{lemma}
\begin{proof}
Suppose, on the contrary, that there are $\lambda_{2}> \lambda_{1}^{*}$ and $(t_{0},s_{0}) \in (0,1)\times (0,1)$ such that $G_{\lambda_{2}}(t_{0},s_{0})=0$, where $G_{\lambda_{2}}$ is the Green's function associated to problem (\ref{pr1}) for $\lambda=\lambda_{2}$.

If $s_{0}\leq t_{0}$,
define the function $V:I\rightarrow \mathbb{R}$, as $V(t)=G_{\lambda_{2}}(t,s_{0})$.

It is clear, from \eqref{G}, that $V$ is not identically zero on $[t_{0},1]$ and solves the following problem
\begin{eqnarray}
\begin{cases}
D^{\alpha} V(t)-\lambda_{2} V(t)=0,  \ \ \ \ \ t \in (t_{0},1),\\ V(t_0)=V'(1)=0.
\end{cases} \label{prr}
\end{eqnarray}

In particular, $\lambda_{2}$ is an eigenvalue of problem (\ref{prr}).\\

Arguing as in the beginning of this section, it is immediate to verify that the eigenvalues $\overline{\lambda}_{n}$ of problem (\ref{prr}) are given as the zeros of the following equality 
\begin{equation}\label{01}
E_{\alpha,\alpha-1}(\lambda)\, E_{\alpha,\alpha-1}(\lambda\, t_0^\alpha)=t_0\,E_{\alpha,\alpha}(\lambda\,t_0^\alpha)\, E_{\alpha,\alpha-2}(\lambda).
\end{equation}

By means of numerical approach, one can verify (See Figures \ref{eigenvalue-1} and \ref{eigenvalue-2}) that all the zeros $\overline{\lambda}$ of previous equation for $t_0>0$ satisfy that $\overline{\lambda} <\lambda_{1}^{*}$, which is a contradiction with the choice of $\lambda_2$.

\begin{figure}[h!]
	\centering
	\includegraphics{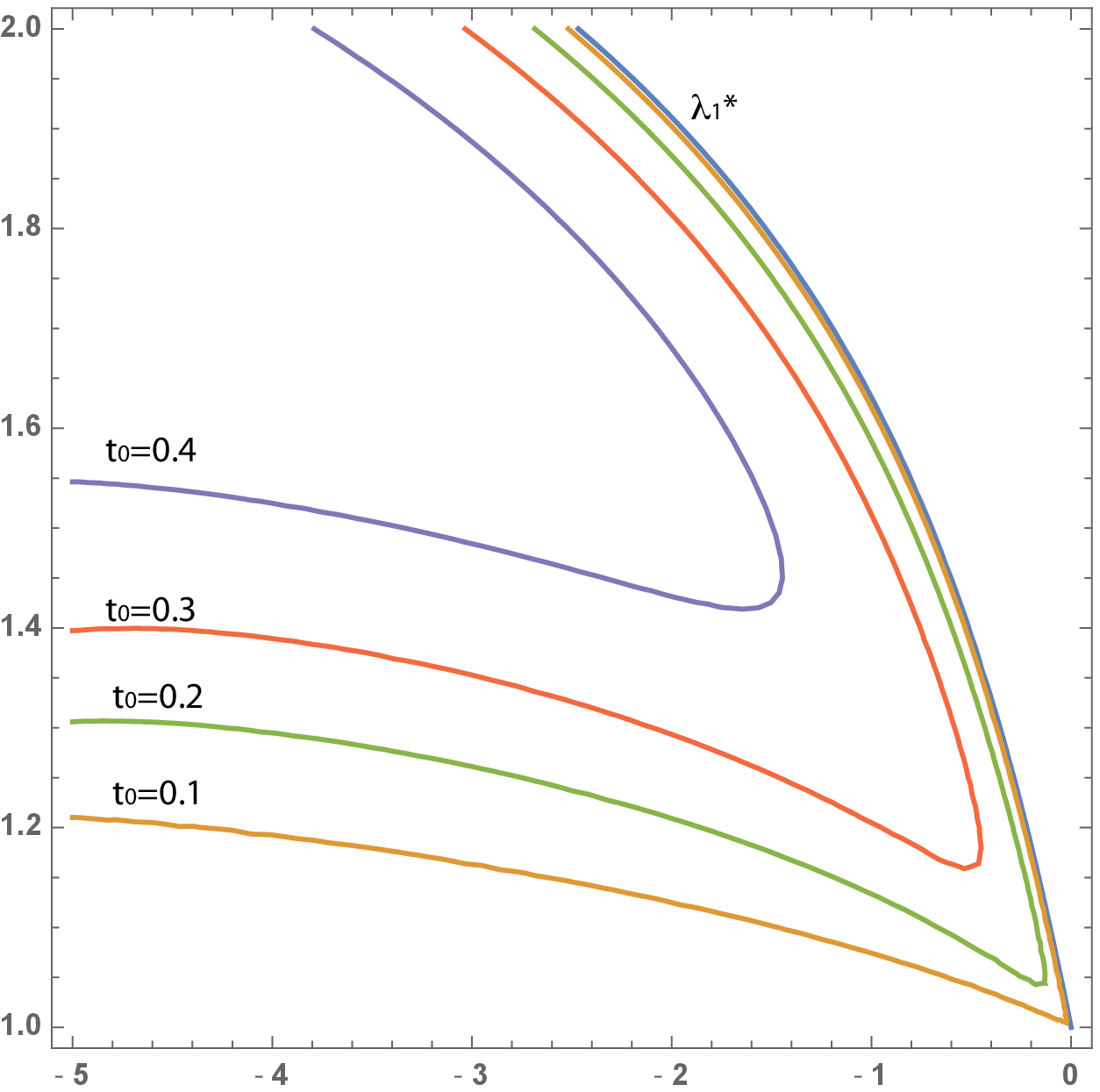}
	\caption{Graph of the first zeros of \eqref{01} for some $0<t<1/2$ and $1< \alpha \le 2$.\label{eigenvalue-1}}
\end{figure}

\begin{figure}[h!]
	\centering
	\includegraphics{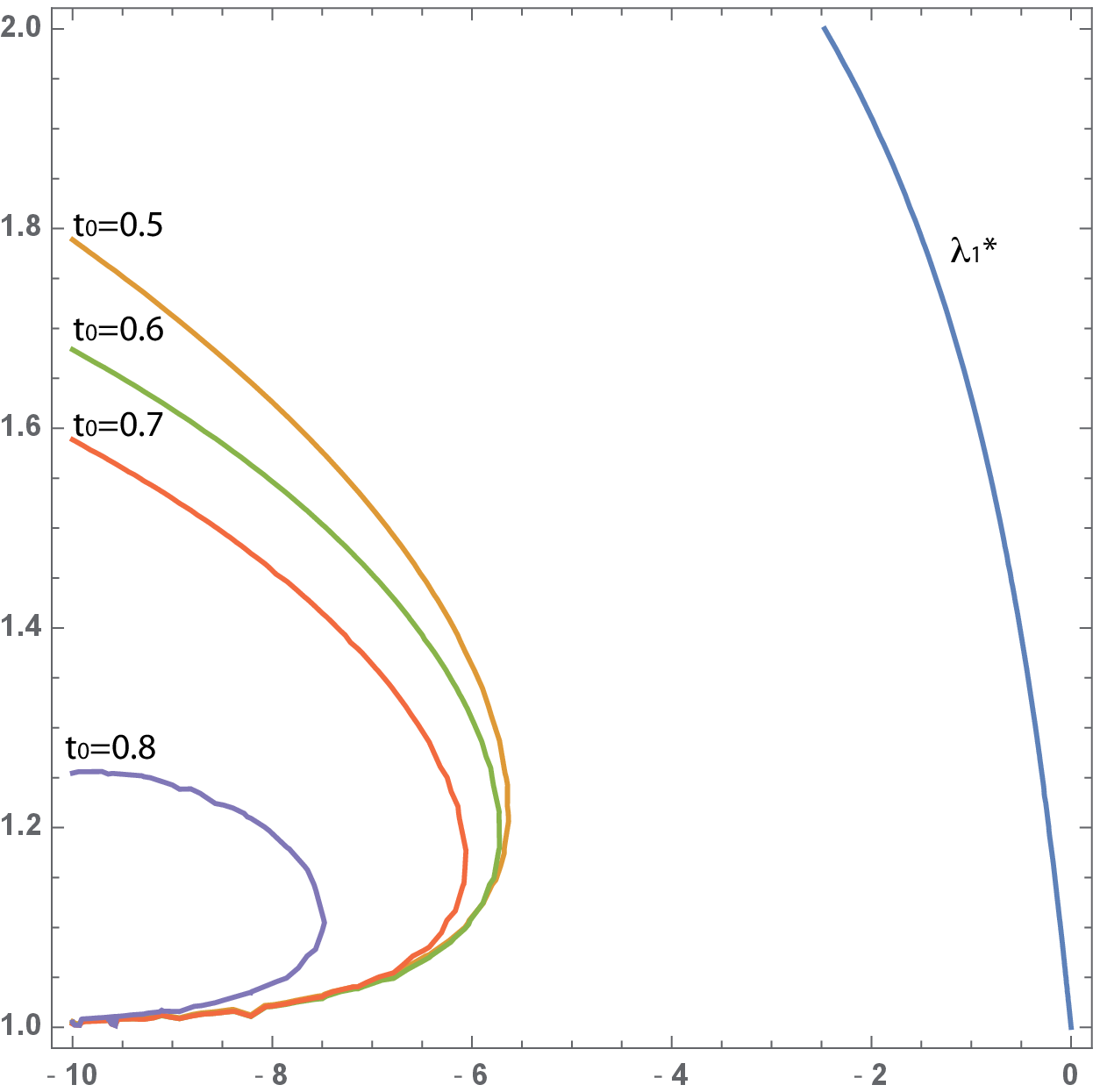}
	\caption{Graph of the first zeros of \eqref{01} for some $1/2<t<1$ and $1< \alpha \le 2$.\label{eigenvalue-2}}
\end{figure}

Now, suppose that $$ G_{\lambda_{2}} (t_{0},s_{0})=0 \ \ \textrm{ for} \ \ t_{0} \leq s_{0}.$$

So, we have that $V(t)=G_{\lambda_{2}} (t,s_{0}) \neq 0$ on $[0,t_{0}]$, satisfies 
\begin{eqnarray}
\begin{cases}
D^{\alpha} V(t)-\lambda_{2} V(t)=0,  \ \ \ \ \ t \in (0,t_{0}),\\
\lim\limits_{t\rightarrow 0^+} t^{2-\alpha} V(t)=V(t_{0})=0.
\end{cases} \label{prrd}
\end{eqnarray}

As it is showed in \cite{ref2}, the eigenvalues of this problem are given as the roots of the following equality:
$$ E_{\alpha,\alpha}(\tilde{\lambda}_{n} t_{0}^\alpha)=0.$$

It is known \cite{ref1} that such equation has a finite number of real roots, all of them negative, for all $\alpha \in (1,2)$.

Moreover, it is not difficult to verify that the biggest root $\lambda_{1}$, of $E_{\alpha,\alpha}(\lambda)=0$, satisfies that $\lambda_{1} < \lambda_{1}^{*}$ (See Figure \ref{fig1})

\begin{figure}[h!]
	\centering
	\includegraphics{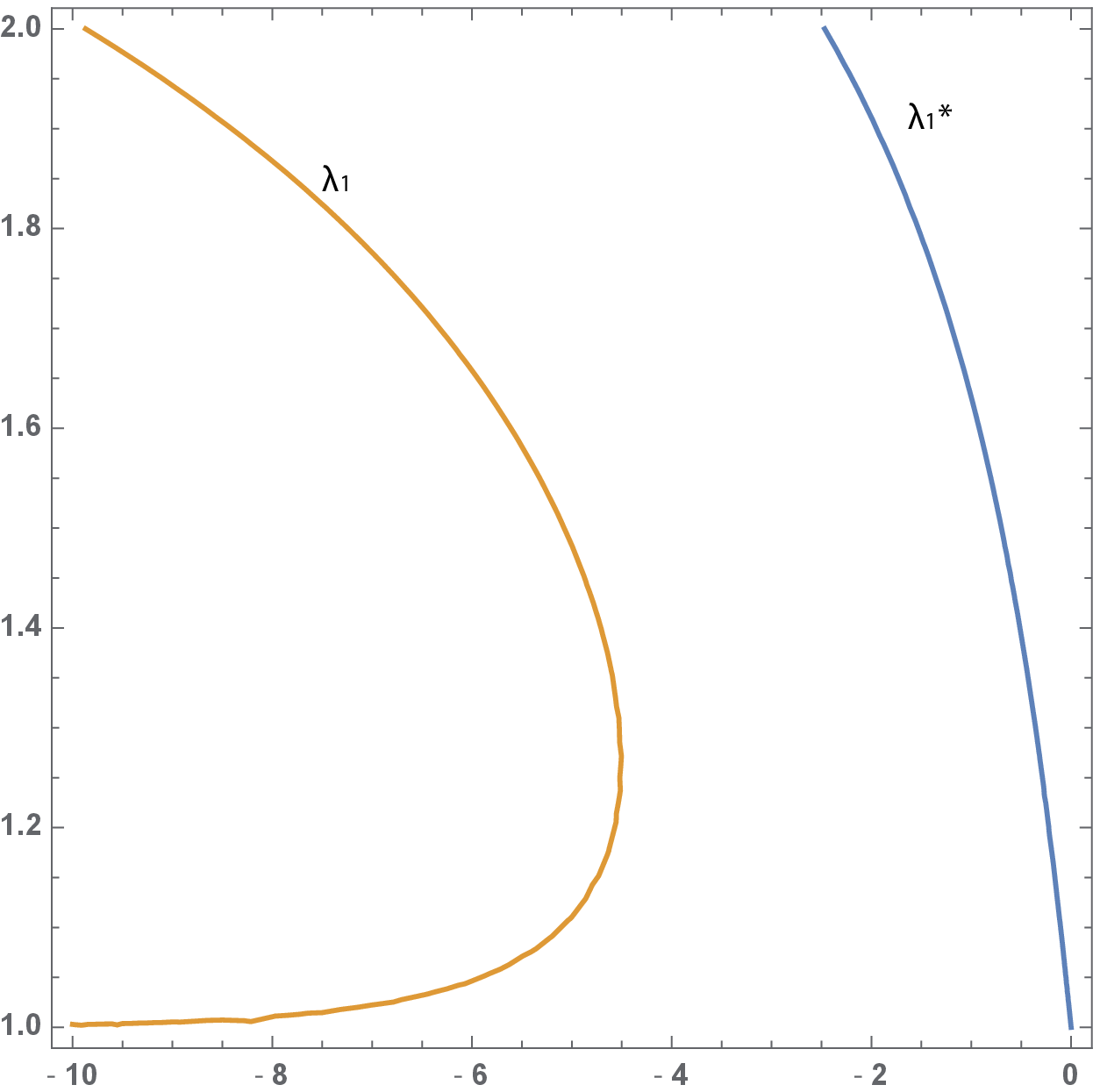}
	\caption{Graph of $\lambda_1$  and $\lambda_1^*$  for $1< \alpha \le 2$.\label{fig1}}
\end{figure}

In particular, the first eigenvalue of (\ref{prrd}) $\tilde{\lambda}_{1}=\frac{\lambda_{1}}{t_{0}^{\alpha}}$ satisfies that 
$$ \tilde{\lambda}_{1} < \lambda_{1} < \lambda_{1}^{*} < \lambda_{2}.$$

Which contradicts the fact that $\tilde{\lambda}_{1}$ is the biggest eigenvalue of problem (\ref{prrd}).

As a consequence, we have proved that if $\lambda > \lambda_{1}^*$ then the Green's function is positive on $(0,1) \times (0,1)$.

If $\lambda < \lambda_{1}^*$ is not an eigenvalue of problem \eqref{pr1}, then, since $\lambda (1-s)^\alpha$ attains all the values of the interval $[\lambda, \lambda (1-s)^\alpha]$, from expression \eqref{G}, it is immediate to verify that the Green's function $G$ changes its sign on the triangle $0 \le t < s \le 1$.
\end{proof}

We remark that for any arbitrary bounded interval $[a,b]$, we can obtain the following result of positivity of Green's function and the validity of a comparison result for the mixed problem.
\begin{corollary}\label{coro-(a,b)}
Let $a<b$, $1< \alpha \leq 2$ and $y \in C([a,b])$. Let $G$ be the Green's function associated to the following problem
\begin{eqnarray}
\label{e-ab}
\begin{cases}
D^{\alpha} u(t)-\lambda u(t)+ y(t)=0,  \quad t \in (a,b),\\
\lim\limits_{t\rightarrow a^+} (t-a)^{2-\alpha}u(t)=u'(b)=0,
\end{cases} 
\end{eqnarray}
and $\lambda_{1}^{*}$ be the first negative root of $E_{\alpha,\alpha-1}(\lambda)=0$. Then 
$$ G(t,s)>0 \ \ for \ all \ t,s\in(a,b)\ \ if \ and \ only \ if \ \ \ \ \lambda > \frac{\lambda_{1}^{*}}{(b-a)^\alpha}.$$

Moreover, if $y\ge 0$, $y \not \equiv 0$, on $(a,b]$, the unique solution $u$ of Problem \eqref{e-ab} satisfies that $u > 0$ on $(a,b]$.
\end{corollary}

In our approach, we need to prove the following sharp inequalities for the Green's function. 
\begin{lemma}\label{lm3}
Let $G$ be the Green's function associated to problem (\ref{pr1}) given in (\ref{G}), $1< \alpha \leq 2$ and $\lambda > \lambda_{1}^{*}$.\ Then there exists a positive constant $M$ and a continuous function $m$ such that $m(t) > 0$ on $(0,1]$ and $m(0)=0$, for which the following inequalities are fulfilled:
\begin{equation}\label{prop}
m(t)\leq \frac{t^{2-\alpha}G(t,s) }{s (1-s)^{\alpha-2}} \leq M, \ \ \ for \ all \ t,s\in(0,1).
\end{equation}
\end{lemma}
\begin{proof}
Define $$H(t,s)=\frac{t^{2-\alpha}(1-s)^{2-\alpha}G(t,s) }{s}$$
which is a continuous function on $I\times (0,1]$.\\
 In addition, for $t \in I$, we have
 \begin{eqnarray*}
\lim_{s \rightarrow 0^{+}} H(t,s) &= & \lim_{s \rightarrow 0^{+}} t^{2-\alpha} \frac{e_{\alpha}^{\lambda t} E_{\alpha,\alpha-1}(\lambda(1-s)^\alpha)-(1-s)^{2-\alpha}E_{\alpha,\alpha-1}(\lambda)e_{\alpha}^{\lambda (t-s)} }{s \ E_{\alpha,\alpha-1}(\lambda)}\\
&= &  \frac{t^{\alpha} \lambda \left( (\alpha-1) E_{\alpha,2 \alpha}(\lambda t^\alpha) -E_{\alpha,2 \alpha-1}(\lambda t^\alpha)          \right)-((1-\alpha)+(\alpha-2)t)E_{\alpha,\alpha-1}(\lambda ) E_{\alpha,\alpha}(\lambda t^\alpha)}{E_{\alpha,\alpha-1}(\lambda)} \\
& & -\frac{t \lambda E_{\alpha,\alpha}(\lambda t^\alpha) \left(  E_{\alpha,2 \alpha-2}(\lambda ) -(\alpha-2)E_{\alpha,2 \alpha-1}(\lambda )          \right)}{E_{\alpha,\alpha-1}(\lambda)}\\
&= & L(t).
\end{eqnarray*}
So, $L$ exists, is finite and $L(t)>0$ for all $t \in (0,1]$. \\
Thus, $H$ is extended by continuity to $I\times I$ as follows
\begin{eqnarray*}
\tilde{H}(t,s)=
\begin{cases}
H(t,s),  \ \ \  \ \ t, s \in I\times (0,1],\\
L(t),  \ \ \  \ \ \ \ \  t \in I, \ s=0.
\end{cases} 
\end{eqnarray*}
As a direct consequence, we have that $$m(t)=\min_{s\in I} \tilde{H}(t,s)$$
is a continuous function on $I$, $m(0)=0$ and $m(t) >0$ for all $t \in (0,1]$. Moreover, for all $t,s \in I$, we obtain 
$$ 0 \leq m(t) \leq \tilde{H}(t,s) \leq M,$$
where $M= \max\limits_{(t,s)\in I\times I} \tilde{H}(t,s)$.
\end{proof}

\section{Existence and uniqueness of positive solutions}
In this section, we will be concerned with the existence and uniqueness of positive solutions to the nonlinear problem (\ref{pr}). To this end, we apply a variant of the classical Krasnoselskii's fixed point theorem proved in \cite{ref5}.

Let $E$ be a real Banach space ordered by the cone $E_{+}$.\ An ordered interval is defined as
$$[x,y]=\lbrace z \in E: \ x\leq z \leq y \rbrace .$$ 

For any $r>0$, we denote $\Omega_{r}=\lbrace x \in E: \ \Vert x \Vert <r \rbrace$ and $\partial\Omega_{r}=\lbrace x \in E: \ \Vert x \Vert =r \rbrace$. Letting $u_{0} \in E_{+}$ with $\Vert u_{0} \Vert \leq 1$, define the subcone $\textit{P}_{u_{0}}$ on the Banach space $E$ as follows 
\begin{equation}\label{cone}
\textit{P}_{u_{0}}= \lbrace  x \in E_{+}, \ x \geq \Vert x \Vert u_{0} \rbrace.
\end{equation}

Our results are based on the following fixed point theorem
\begin{theorem}\label{th00}(\cite[Theorem 2.1]{ref5})
Assume that $E$ is an ordered Banach space with the order cone $E_{+}$.\ Let $0 \leq u_{0} \leq \varphi$ be such that $\Vert u_{0} \Vert \leq 1$, $\Vert \varphi \Vert =1$ satisfying the condition 
\begin{equation}\label{con}
 if \ u \in E_{+} \; \mbox{and}\; \Vert u \Vert \leq 1, \ \ then \ \ u \leq \varphi .
 \end{equation}
If there exist positive numbers $0< a<b$ such that $T: \textit{P}_{u_{0}}\cap(\overline{\Omega_{b}} \backslash \Omega_{a})\longrightarrow 
\textit{P}_{u_{0}}$ is a completely continuous operator and the conditions
$$ \Vert Tu \Vert_{u \in  [a u_{0}, a \varphi]} \leq a \ \ and \ \ \Vert Tu \Vert_{u \in  [b u_{0}, b \varphi]} \geq b                  $$
or 
$$  \Vert Tu \Vert_{u \in  [a u_{0}, a \varphi]} \geq a \ \ and \ \ \Vert Tu \Vert_{u \in  [b u_{0}, b \varphi]} \leq b                       $$
are satisfied, then $T$ has at least one fixed point $u \in [a u_{0}, b \varphi]$.
\end{theorem}

To this end, define the operator $T: C_{2-\alpha}(I) \longrightarrow  C_{2-\alpha}(I)$ by
\begin{equation}\label{oper}
Tu(t)= \int_{0}^{1} G(t,s) f(s,s^{2-\alpha}u(s))ds, \ \ \ \ \ t \in I.
\end{equation}
where $G$ is defined in \eqref{G}.

Consider the cone $ \textit{P}_{u_{0}} \subset C_{2-\alpha}(I)$, defined as:
\begin{equation}\label{cone1}
 \textit{P}_{u_{0}}:= \lbrace  u \in C_{2-\alpha}(I), \ u(t) \geq u_{0}(t) \Vert u \Vert_{2-\alpha}, \ t \in I \rbrace
\end{equation}
with $u_{0}(t)= t^{\alpha-2} \frac{m(t)}{M}$, $t \in I$, and $m(t)$ and $M$ defined in Lemma \ref{lm3}.

It is clear that  for $t \in I$, we have that
$t^{2-\alpha} u_{0}(t)= \frac{m(t)}{M} \in C(I)$, i. e. $ u_{0} \in C_{2-\alpha}(I)$. Moreover $\Vert u_{0} \Vert_{2-\alpha} = \max \lbrace  t^{2-\alpha} u_{0}(t), \ \ t \in I      \rbrace \leq 1$. 

In the remainder of the paper, we assume the following hypothesis:

\begin{itemize}
	\item [$(H)$]  $f : I\times [0, \infty)\longrightarrow [0,\infty)$ is a continuous function.
\end{itemize}

Hereinafter, we use the following notations
$$ f_{0}= \lim_{u\rightarrow 0^+} \lbrace  \min_{t\in I} \frac{f(t,u)}{u}  \rbrace \ \ \textrm{and} \ \  f_{\infty}= \lim_{u\rightarrow \infty} \lbrace  \min_{t\in I} \frac{f(t,u)}{u}  \rbrace, $$  
and 
$$ f^{0}= \lim_{u\rightarrow 0^+} \lbrace  \max_{t\in I} \frac{f(t,u)}{u}  \rbrace \ \ \textrm{and} \ \  f^{\infty}= \lim_{u\rightarrow \infty} \lbrace  \max_{t\in I} \frac{f(t,u)}{u}  \rbrace. $$ 

\begin{lemma}\label{op1}
Assume that (H) holds and $\lambda > \lambda_{1}^{*}$.\ Then $T: \textit{P}_{u_{0}} \longrightarrow \textit{P}_{u_{0}}$ is a completely continuous operator.
\end{lemma}
\begin{proof}
Notice from Lemmas \ref{lm1} and \ref{lm2} that, for $u \in \textit{P}_{u_{0}}$, we have that $Tu(t) \geq 0$, for all $ t \in I$ and $\Vert Tu \Vert_{2-\alpha} < \infty$.\\
Now, let $u \in \textit{P}_{u_{0}}$ then, for all $t \in I$,
\begin{eqnarray*}
t^{2-\alpha} Tu(t) &= & \int_{0}^{1}t^{2-\alpha} G(t,s)f(s,s^{2-\alpha}u(s))ds\\
& \geq & m(t) \int_{0}^{1} s (1-s)^{\alpha-2}f(s,s^{2-\alpha}u(s))ds\\
& \geq & \frac{m(t)}{M} \int_{0}^{1} \max\limits_{t \in I}\lbrace t^{2-\alpha} G(t,s)\rbrace f(s,s^{2-\alpha}u(s))ds\\
& \geq & \frac{m(t)}{M} \max\limits_{t \in I} \left\{ t^{2-\alpha} \int_{0}^{1} G(t,s)f(s,s^{2-\alpha}u(s))ds\right\} \\
& = & \frac{m(t)}{M} \Vert Tu \Vert_{2-\alpha}.
\end{eqnarray*}

Next, we will show that $T$ is uniformly bounded. Let $\Omega_{1}\subset \textit{P}_{u_{0}}$ be bounded set of $\textit{P}_{u_{0}}$, then there exists a positive constant $L>0$ such that $\Vert u \Vert_{2-\alpha} \leq L$. \\

Let $$ M_{0}= \max\limits_{t \in I, x \in [0,L]} \vert f(t,x) \vert +1.$$

Then, by $(H)$ and Lemma \ref{lm3}, we have for all $u \in \Omega_{1}$ and $ t \in I$
\begin{eqnarray*}
\vert t^{2-\alpha} Tu(t) \vert & \leq & M_{0} \int_{0}^{1}t^{2-\alpha} G(t,s)ds\\
& \leq & M_{0} M \int_{0}^{1} s (1-s)^{\alpha-2} ds =\frac{M_{0} M }{\alpha (\alpha-1)}.
\end{eqnarray*}

Hence, $T( \Omega_{1})$ is bounded.\\

Now, let us prove that $T( \Omega_{1})$ is equicontinuous in $C_{2-\alpha}(I)$.\\

For any $t_{1}, t_{2} \in I$ such that $t_{1}\leq t_{2}$ and $u \in \Omega_{1}$, we have
\begin{eqnarray*}
\vert t_{2}^{2-\alpha} Tu(t_{2})- t_{1}^{2-\alpha} Tu(t_{1}) \vert &= &  \big\vert  \int_{0}^{1}t_{2}^{2-\alpha} G(t_{2},s)f(s,s^{2-\alpha}u(s))ds- \int_{0}^{1}t_{1}^{2-\alpha} G(t_{1},s)f(s,s^{2-\alpha}u(s))ds \\
& \leq & \scriptstyle{\int_{0}^{t_{1}} \left|  \frac{ t_{2} E_{\alpha, \alpha}(\lambda t_{2}^{\alpha}) E_{\alpha, \alpha-1}(\lambda (1-s)^\alpha)}{E_{\alpha, \alpha-1}(\lambda)}- \frac{ t_{1} E_{\alpha, \alpha}(\lambda t_{1}^{\alpha}) E_{\alpha, \alpha-1}(\lambda (1-s)^\alpha)}{ E_{\alpha, \alpha-1}(\lambda)} \right| (1-s)^{\alpha -2} f(s,s^{2-\alpha}u(s))ds }\\
&  &+   \scriptstyle{\int_{0}^{t_{1}} \left|t_{1}^{2-\alpha} (t_{1}-s)^{\alpha -1}  E_{\alpha, \alpha}(\lambda (t_{1}-s)^{\alpha}) -t_{2}^{2-\alpha}(t_{2}-s)^{\alpha-1}  E_{\alpha, \alpha}(\lambda (t_{2}-s)^{\alpha}) \right| f(s,s^{2-\alpha}u(s))ds } \\
&  & +\scriptstyle{\int_{t_{1}}^{t_{2}} \left|  \frac{ t_{2} E_{\alpha, \alpha}(\lambda t_{2}^{\alpha}) E_{\alpha, \alpha-1}(\lambda (1-s)^\alpha)}{E_{\alpha, \alpha-1}(\lambda)}- \frac{ t_{1} E_{\alpha, \alpha}(\lambda t_{1}^{\alpha}) E_{\alpha, \alpha-1}(\lambda (1-s)^\alpha)}{ E_{\alpha, \alpha-1}(\lambda)} \right| (1-s)^{\alpha -2} f(s,s^{2-\alpha}u(s))ds }\\
 & & +\scriptstyle{t_{2}^{2-\alpha} \int_{t_{1}}^{t_{2}}  (t_{2}-s)^{\alpha -1}  E_{\alpha, \alpha}(\lambda (t_{2}-s)^{\alpha})       f(s,s^{2-\alpha}u(s))ds}\\
 & & + \scriptstyle{\int_{t_{2}}^{1} \left|  \frac{ t_{2} E_{\alpha, \alpha}(\lambda t_{2}^{\alpha}) E_{\alpha, \alpha-1}(\lambda (1-s)^\alpha)}{E_{\alpha, \alpha-1}(\lambda)}- \frac{ t_{1} E_{\alpha, \alpha}(\lambda t_{1}^{\alpha}) E_{\alpha, \alpha-1}(\lambda (1-s)^\alpha)}{ E_{\alpha, \alpha-1}(\lambda)} \right| (1-s)^{\alpha -2} f(s,s^{2-\alpha}u(s))ds }.
\end{eqnarray*}

Note that the function 
$ (t,s) \mapsto \frac{ t E_{\alpha, \alpha}(\lambda t^{\alpha}) E_{\alpha, \alpha-1}(\lambda (1-s)^\alpha)}{E_{\alpha, \alpha-1}(\lambda)}$ is uniformly continuous on $I\times I$.\\

Then, for $\varepsilon >0$ there exists $\delta >0$ such that if $\vert t_{1}-t_{2}\vert < \delta$, we obtain that the first integral is bounded by 
\begin{equation*}
M_{0} \varepsilon \int_{0}^{t_{1}}(1-s)^{\alpha-2} ds= M_{0} \varepsilon \frac{(1-(1-t_{1})^{\alpha-1})}{\alpha-1}\leq \frac{M_{0}}{\alpha-1}\varepsilon.
\end{equation*}

Arguing in an analogous way with the third and the last terms of the inequalities above, we get that the integrals are bounded by 
\begin{equation*}
M_{0} \varepsilon \int_{t_{1}}^{t_{2}}(1-s)^{\alpha-2} ds= M_{0} \varepsilon \frac{((1-t_{1})^{\alpha-1}-(1-t_{2})^{\alpha-1})}{\alpha-1}\leq \frac{M_{0}}{\alpha-1}\varepsilon
\end{equation*}
and 
\begin{equation*}
M_{0} \varepsilon \int_{t_{2}}^{1}(1-s)^{\alpha-2} ds= M_{0} \varepsilon \frac{(1-t_{2})^{\alpha-1}}{\alpha-1}\leq \frac{M_{0}}{\alpha-1}\varepsilon,
\end{equation*}
respectively.\\

Moreover, by the uniformly continuity  on $I\times I$ of function 
\[ (t,s) \mapsto t^{2-\alpha}(t-s)^{\alpha-1}  E_{\alpha, \alpha}(\lambda (t-s)^{\alpha}),\]
 we deduce that the second integral is bounded by $M_{0} \varepsilon$.
 
On the other hand, the function $(t_{2}-s)^{\alpha -1}  E_{\alpha, \alpha}(\lambda (t_{2}-s)^{\alpha})$ is continuous and we have, for a fixed $\lambda$, that 
$$ \vert  E_{\alpha, \alpha}(\lambda (t-s)^{\alpha})\vert \leq M_{1}, \ \forall \ (t,s) \in I\times I, \ \ \alpha \in(1,2].$$

So, 
\begin{equation*}
R= \big\vert \int_{t_{1}}^{t_{2}}  (t_{2}-s)^{\alpha -1}  E_{\alpha, \alpha}(\lambda (t_{2}-s)^{\alpha}) f(s,s^{2-\alpha}u(s))ds \big\vert \leq M_{0} M_{1} \int_{t_{1}}^{t_{2}}  (t_{2}-s)^{\alpha -1} ds=  \frac{ M_{0} M_{1}}{2} (t_{2}-t_{1})^\alpha.   
\end{equation*}

Thus, there exists $\delta >0$ such that if $\vert t_{2}-t_{1}\vert < \delta$, then $R \leq \varepsilon$.\\

So, for $\varepsilon >0$, there exists $\delta >0$ such that if  $\vert t_{2}-t_{1}\vert < \delta$ we deduce that
\begin{equation*}
\vert t_{2}^{2-\alpha} Tu(t_{2})- t_{1}^{2-\alpha} Tu(t_{1}) \vert \leq
\left( \frac{3M_{0}}{\alpha-1} +M_{0}+1  \right)  \varepsilon.
\end{equation*}

Then, operator $t^{2-\alpha} Tu(t)$ is equicontinuous in $C(I).$ And so, $T(\Omega_{1})$ is equicontinuous in $C_{2-\alpha}(I)$.\ By Ascoli's theorem $T (\Omega_{1})$ is a relatively compact in $C_{2-\alpha}(I)$.\\

As a consequence, $T: \textit{P}_{u_{0}} \longrightarrow \textit{P}_{u_{0}}$ is completely continuous operator. \\This completes the proof.
\end{proof}\\

 Now, we are able to show the following existence result.
\begin{theorem}\label{th01}
Assume that condition (H) holds. In addition, if one of the following conditions is satisfied 
\begin{enumerate}
\item[(1)]$f_{0}=\infty$ and $f^{\infty}=0$.
\item[(2)]$f^{0}=0$ and $f_{\infty}=\infty$.
\end{enumerate}
Then, for all $1< \alpha \leq 2$ and $\lambda > \lambda_{1}^{*}$, Problem (\ref{pr}) has at least one positive solution in $\textit{P}_{u_{0}}$.
\end{theorem}
\begin{proof} In this case we take $\varphi(t)=t^{\alpha-2}$. Clearly, $\varphi \in C_{2-\alpha}(I)$ and $\Vert \varphi  \Vert_{2-\alpha}=1$. Moreover, it follows that $0\leq u_{0} \leq \varphi$ and  condition (\ref{con}) is trivially fulfilled.
	
	The rest of the proof is essentially the one given in \cite[Theorem 4.2]{ref2}, and we omit them.
\end{proof}\\

Now, to prove the uniqueness of positive solution, we need to assume that function $f$ satisfies the following condition:\\
\begin{itemize}
	\item [$(H^*)$]  There exists a constant $K >0$ such that 
	\[
	\vert f(t,u)-f(t,v) \vert \leq K \vert u-v \vert \ \ for \ each \ t \in I \ and \ u,v \in \textit{P}_{u_{0}}.
	\] 
\end{itemize}

So, the uniqueness result is the following.
\begin{theorem}\label{thuq}
Assume that (H) and (H$^{*}$) are fulfilled. Then Problem (\ref{pr}) has a unique solution in $\textit{P}_{u_{0}}$ provided that 
\begin{equation}\label{fin}
\frac{K \ M}{\alpha(\alpha-1)} \ <1.
\end{equation}
\end{theorem}
\begin{proof}
To prove the previous result, we apply the Banach fixed point theorem \cite{ref7}.\\
So, let us show that $T$ is a contraction operator in $\textit{P}_{u_{0}}$.\\
Let $t \in I$ and $u,v \in \textit{P}_{u_{0}}$, we have 
\begin{eqnarray*}
\vert t^{2-\alpha} Tu(t)-t^{2-\alpha} Tv(t) \vert & \leq & t^{2-\alpha} \int_{0}^{1} G(t,s) \vert f(s, s^{2-\alpha}\ u(s))-f(s, s^{2-\alpha}\ v(s)) \vert ds\\
&\leq & K M t^{2-\alpha} \int_{0}^{1} s (1-s)^{\alpha-2} \vert s^{2-\alpha}\ u(s)- s^{2-\alpha}\ v(s) \vert ds\\
&\leq & K M \int_{0}^{1} s (1-s)^{\alpha-2} ds\ \Vert u-v \Vert_{2-\alpha}\\
&= & \frac{K \ M}{\alpha(\alpha-1)}  \Vert u-v \Vert_{2-\alpha}.
\end{eqnarray*}
Therefore, $T$ is a contraction operator in $\textit{P}_{u_{0}}$ and we deduce that $T$ has a unique fixed point $u \in \textit{P}_{u_{0}}$, which is a unique solution of problem (\ref{pr}) in such subcone.
\end{proof}

\section{Existence of solutions via nondecreasing operators}

In the sequel we will prove the existence of positive solutions of problem \eqref{pr} by using the following fixed point proved in \cite{ref11} for nondecreasing operators on ordered Banach spaces.
\begin{theorem}\label{th*}(\cite{ref11})
	Let $X$ be a real Banach space, $K$ a normal and solid cone that induces in $X$ the order $\preceq$, and $T: K\rightarrow K$ a nondecreasing and completely continuous operator. Define $$S = \lbrace u \in K: Tu \preceq u  \rbrace $$
	and suppose that 
	\begin{itemize}
		\item[(i)] There exists $\overline{u} \in S$ such that $\overline{u} \in Int(K)$.
		\item[(ii)] $S$ is bounded.
	\end{itemize}
	Then there exists $u \in K$, $u\neq 0$, such that $u=Tu$.
\end{theorem}
First of all, we define the new cone $P^{*}$ as follows
\begin{equation*}
P^{*} = \lbrace  u \in C_{2-\alpha}(I):\ u(t) \ge 0 \  \textrm{ for all} \ t \in(0,1], \ t^{2-\alpha}u(t) \ge \frac{m_{0}}{M} \Vert u \Vert_{2-\alpha} \ \textrm{on}\ [c_{1},1]                      \rbrace,
\end{equation*}
with $m_{0}= \min\limits_{t\in [c_{1},1]} m(t)$, and $c_{1} \in (0,1)$.\\
The existence result is the following one.
\begin{theorem}\label{th**}
	Suppose that (H) holds, $\lambda > \lambda_1^*$, and the following assumptions hold:
	\begin{itemize}
		\item[(H$_{1}$)] $f(t,\cdot)$ is nondecreasing for each $t \in I$.
		\item[(H$_{2}$)] $f_{\infty}= \lim\limits_{u\rightarrow \infty} \lbrace  \min\limits_{t\in [c_{1},1]} \frac{f(t,u)}{u}  \rbrace =\infty $.
		\item[(H$_{3}$)] There exists $\overline{u} \in C_{2-\alpha}(I)$ such that $D^{\alpha} \overline{u} \in C(0,1]\cap L^1(0,1)$ with \\ $\lim\limits_{t\rightarrow 0^+} t^{2-\alpha}\overline{u}(t)=0, \ \ \overline{u}'(1)=0$ and moreover for $\lambda > \lambda_{1}^{*}$,
		\[  D^{\alpha} \overline{u}(t)- \lambda \overline{u}(t)+ f(t,t^{2-\alpha}\overline{u}(t)) > 0 \ \ \ for \ t \in (0,1).
		\]
	\end{itemize}
	Then Problem (\ref{pr}) has a positive solution on $(P^*)$.
\end{theorem}
\begin{proof}
	Notice that in $P^*$  induced the following order in $C_{2-\alpha}(I)$:
	$$	u,v \in C_{2-\alpha}(I), \; u \preceq v \; \mbox{ if and only if } $$
	\[
	v(t) \ge u(t) \, \mbox{for all } t \in (0,1], \;  \mbox{and  } t^{2-\alpha}(v-u)(t) \ge \frac{m_{0}}{M} \Vert v-u \Vert_{2-\alpha} \mbox{for all } t \in [c_{1},1]  .                
	\]
	First, let us prove that $T:P^{*}\rightarrow P^{*}$ is a nondecreasing operator. Using (H$_{1}$), for $t \in I$ and $u,v \in P^{*}$ such that $u \preceq v$, we have
	\begin{eqnarray*}
		t^{2-\alpha}Tu(t)&= & t^{2-\alpha} \int_{0}^{1} G(t,s) f(s,s^{2-\alpha}u(s))ds \\
		& \leq & t^{2-\alpha} \int_{0}^{1} G(t,s) f(s,s^{2-\alpha}v(s))ds\\
		& = & t^{2-\alpha}Tv(t).
	\end{eqnarray*}
	
	Moreover, for every $t \in [c_1,1]$, the following inequalities hold:
	
	\begin{eqnarray*}
		t^{2-\alpha} (T v(t)-Tu(t))& = & t^{2-\alpha} \int_{0}^{1} G(t,s) \left(f(s,s^{2-\alpha}v(s))-f(s,s^{2-\alpha}u(s))\right)ds \\
		& \geq & m(t) \int_{0}^{1} s (1-s)^{\alpha -2} \left(f(s,s^{2-\alpha}v(s))-f(s,s^{2-\alpha}u(s))\right)ds\\
		& \geq & \frac{m(t)}{M}  \int_{0}^{1} \max\limits_{t \in I} \lbrace t^{2-\alpha} G(t,s)\rbrace \left(f(s,s^{2-\alpha}v(s))-f(s,s^{2-\alpha}u(s))\right)ds\\
		& \geq &  \frac{m(t)}{M}  \max\limits_{t \in I} \lbrace  \int_{0}^{1}  t^{2-\alpha} G(t,s)\left(f(s,s^{2-\alpha}v(s))-f(s,s^{2-\alpha}u(s))\right)ds \rbrace \\
		&\geq & \frac{m_{0}}{M} \Vert Tv-Tu \Vert_{2-\alpha}.
	\end{eqnarray*}

	Therefore $T$ is a nondecreasing operator on $P^*$.
	
	By Lemma \ref{op1}, we know that $T:P^{*}\rightarrow P^{*}$ is a completely continuous operator.
	
	Now, we shall prove that $\overline{u} \in S=\lbrace u \in P^{*} : Tu \preceq u  \rbrace$ and $\overline{u} \in Int(P^*)$.
	
	From (H$_{3}$), there exists a nonnegative function $\sigma \in C(0,1]\cap L^1(0,1)$ such that 
	\[
	D^{\alpha} \overline{u}(t)- \lambda \overline{u}(t) + f(t,t^{2-\alpha}\overline{u}(t))=\sigma(t),\ \ \ for \ t \in (0,1),
	\]
	which is equivalent to 
	\[
	\overline{u}(t)=T\overline{u}(t)+\int_{0}^{1} G(t,s) \sigma(s)ds.
	\]
	So, for $\overline{u} \in P^{*}$ and $t \in I$.  Since $\lambda > \lambda_1^*$, we deduce that
	\[
	\overline{u}(t)>T\overline{u}(t) \quad \mbox{for all  }   t \in (0,1].
	\]
	
	Moreover, for $t \in [c_1,1]$, we obtain  
	\begin{eqnarray*}
		t^{2-\alpha} (\overline{u}(t)-T\overline{u}(t))& = & t^{2-\alpha} \int_{0}^{1} G(t,s) \sigma(s)ds \\
		& \geq & m(t) \int_{0}^{1} s (1-s)^{\alpha -2} \sigma(s)ds\\
		& \geq & \frac{m(t)}{M}  \int_{0}^{1} \max\limits_{t \in I} \lbrace t^{2-\alpha} G(t,s)\rbrace \sigma(s)ds\\
		& \geq &  \frac{m(t)}{M}  \max\limits_{t \in I} \lbrace  \int_{0}^{1}  t^{2-\alpha} G(t,s)\sigma(s)ds \rbrace \\
		&\geq & \frac{m_{0}}{M} \Vert \overline{u}-T\overline{u} \Vert_{2-\alpha}.
	\end{eqnarray*}
	
	Thus, $T\overline{u} \preceq \overline{u}$.
	
	On the other hand, since $t^{2-\alpha}\overline{u}(t)>0$ for all $t \in (0,1]$, we can choose $c_1 \in (0,1)$ such that
	
	\[ t^{2-\alpha}\overline{u}(t) > \frac{m_{0}}{M} \Vert \overline{u} \Vert_{2-\alpha} \quad \mbox{for all $t \in [c_1,1]$,}\]
	and then $\overline{u} \in Int(P^{*})$.
	
	To finish the proof, we must verify that $S$ is bounded.
	
	By $(H_{2})$, for $A= \frac{M}{m_{0}^{2}\int_{c_{1}}^{1}s(1-s)^{\alpha-2}ds}$, there exists $B>0$ such that $f(t,u)> Au$ for all $u>B$ and $t \in [c_{1},1]$.
	
	Let $u \in P^{*}$ be such that $t^{2-\alpha}u(t) >B$ on $[c_1,1]$. Then, for $t \in [c_1,1]$,  we obtain 
	\begin{eqnarray*}
		t^{2-\alpha} Tu(t)& \geq & m(t) \int_{c_{1}}^{1}s(1-s)^{\alpha-2} f(s,s^{2-\alpha}u(s))ds\\
		&> & A\  m(t) \int_{c_{1}}^{1} s(1-s)^{\alpha -2}s^{2-\alpha} u(s)ds \\
		&\geq & A \frac{m_{0}^{2}}{M}  \int_{c_{1}}^{1} s(1-s)^{\alpha -2}ds \ \Vert u \Vert_{2-\alpha}\\
		&= & \Vert u \Vert_{2-\alpha} \geq t^{2-\alpha} u(t).
	\end{eqnarray*}
	
	Thus, $Tu \npreceq  u$, and so $u \notin S$. Then, for each $u \in S$, we have that $t^{2-\alpha}u(t) \leq B.$
	
	Therefore, $S$ is a bounded set provided that 
	\[
	B \geq t^{2-\alpha}u(t) \geq \frac{m_{0}}{M} \Vert u \Vert_{2-\alpha}.
	\]
	
	And finally, from Theorem \ref{th*} we ensure the existence of a fixed point in $P^{*}$ for operator $T$, which is a positive solution of (\ref{pr}).
\end{proof}

\section{Lower and upper solutions}
In this section, we investigate the existence of at least one solution of the non-homogeneous mixed problem 
\begin{eqnarray}
\begin{cases}
D^{\alpha} u(t)-\lambda u(t)+ f(t,t^{2-\alpha}u(t))=0,  \ \ \ \ \ t\in (0,1),\\
\lim\limits_{t \rightarrow 0^+} t^{2-\alpha}u(t)=A,   \ \ \ \                u'(1)=B,
\end{cases} \label{pr*}
\end{eqnarray}
with $\lambda > \lambda_{1}^{*}$, $1<\alpha \leq 2$, $ f \in C(I \times \R)$ and $A \ , B \in \mathbb{R}$.

To this end, we introduce the concept of lower and upper solutions for Problem (\ref{pr*}) as follows
\begin{definition}
Let $\gamma \in C_{2-\alpha}(I)$ be such that $D^\alpha \gamma \in C(0,1]\cap L^1(0,1)$. $\gamma$ is said a lower solution of problem (\ref{pr*}) if it satisfies
\begin{eqnarray}
\begin{cases}
D^{\alpha} \gamma(t)-\lambda \gamma(t)+ f(t,t^{2-\alpha}\gamma(t))\leq 0,  \ \ \ \ \ t \in (0,1),\\
\lim\limits_{t \rightarrow 0^+} t^{2-\alpha}\gamma(t)\leq A,   \ \ \ \                \gamma'(1)\leq B.
\end{cases} \label{pr**}
\end{eqnarray}
$\delta \in  C_{2-\alpha}(I)$, $D^\alpha \delta \in C(0,1]\cap L^1(0,1)$, will be an upper solution of (\ref{pr*}) if the previous inequalities are reversed.
\end{definition}

Before proving the main result of this section, we generalize Corollary \ref{coro-(a,b)}  in the following sense

We remark that for any arbitrary bounded interval $[a,b]$, we can obtain the following result of positivity of Green's function and the validity of a comparison result for the mixed problem.
\begin{lemma}\label{l-(a,b)}
	Let $a<b$, $1< \alpha \leq 2$ and $y \in C(a,b]\cap L^1(a,b)$ be such that $y \ge 0$, $y \not \equiv 0$, on $(a,b]$. Assume that $A_1 \ge 0$ and $B_1 \ge 0$ and $\lambda > \frac{\lambda_{1}^{*}}{(b-a)^\alpha}$, with $\lambda_{1}^{*}$ the first negative root of $E_{\alpha,\alpha-1}(\lambda)=0$. Then the unique solution of problem
	\begin{eqnarray}
	\label{e-non-homogeneous}
		\begin{cases}
			D^{\alpha} u(t)-\lambda u(t)+ y(t)=0,  \quad t \in (a,b),\\
			\lim\limits_{t\rightarrow a^+} (t-a)^{2-\alpha}u(t)= A_1, \quad u'(b)=B_1,
		\end{cases} 
	\end{eqnarray}
	  satisfies that $u > 0$ on $(a,b]$.
\end{lemma}
\begin{proof}
Without loss of generality, we will assume that $a=0$ and $b=1$.

It is clear that the unique solution of problem \eqref{e-non-homogeneous} (with $a=0$ and $b=1$) is given by the following expression
\begin{equation}\label{e-sol-nh}
u(t) = \int_{0}^{1} G(t,s) y(s)ds + A_1\ v_{1}(t)+ B_1 \ v_{2}(t), \ \ \ \ \ 0<t\leq 1,
\end{equation}
where $G$ is the Green's function associated to problem (\ref{pr1}), given by expression \eqref{G}, $v_{1}$ is the unique solution of problem
\begin{eqnarray*}
	\begin{cases}
		D^{\alpha} v_{1}(t)-\lambda v_{1}(t)=0  \ \ \ \ \ t \in (0,1),\\
		\lim\limits_{t\rightarrow 0^+} t^{2-\alpha}v_{1}(t)=1, \ \ \ v'_{1}(1)=0,
	\end{cases} 
\end{eqnarray*}
and $v_{2}$ is the unique solution of problem
\begin{eqnarray*}
	\begin{cases}
		D^{\alpha} v_{2}(t)-\lambda v_{2}(t)=0  \ \ \ \ \ t \in (0,1),\\
		\lim\limits_{t\rightarrow 0^+} t^{2-\alpha}v_{2}(t)=0, \ \ \ v'_{2}(1)=1.
	\end{cases} 
\end{eqnarray*}

From Theorem \ref{th1}, arguing in a similar way as at the beginning of Section 3, we obtain that functions $v_{1}$ and $v_{2}$ are given by the following expressions:
{\tiny
	\begin{eqnarray}
	\label{e-v1}
	v_{1}(t)&= & -\frac{\Gamma(\alpha-1) \left( (\alpha-2) E_{\alpha,\alpha-1}(\lambda)+ \lambda E_{\alpha,2 \alpha -2}(\lambda)-\lambda(\alpha -2)E_{\alpha,2\alpha -1}(\lambda )  \right) }{(\alpha-1)E_{\alpha,\alpha}(\lambda)-\lambda(\alpha -1)E_{\alpha,2\alpha}(\lambda )+\lambda E_{\alpha,2\alpha -1}(\lambda )} t^{\alpha-1} E_{\alpha,\alpha }(\lambda t^{\alpha} )\\
	& & + t^{\alpha-2} \Gamma(\alpha-1) E_{\alpha,\alpha-1}(\lambda t^\alpha), \nonumber
\end{eqnarray}}
and 
\begin{equation}
	\label{e-v2}
v_{2}(t)=\frac{t^{\alpha-1} E_{\alpha,\alpha }(\lambda t^{\alpha} ) }{E_{\alpha,\alpha-1}(\lambda)}.
\end{equation}

So, if there is some $t_0 \in (0,1)$ for which $v_1(t_0)=0$, we have that $v_1$ is a nontrivial solution of Problem \eqref{prr} for $\lambda =\lambda_2$. But, as we have showed in the proof of Lemma \ref{lm2}, this implies that $\lambda < \lambda_1^*$ and we attain a contradiction.

If, $v_2(t_0)=0$ for some  $t_0 \in (0,1)$, we have that $v_2$ is a nontrivial solution of Problem \eqref{prrd} for $\lambda =\lambda_2$. Again, the proof of Lemma \ref{lm2} ensures that $\lambda < \lambda_1^*$.

So, we have that both $v_1$ and $v_2$ are positives on $(0,1]$ and the result holds immediately from expression \eqref{e-sol-nh}.
\end{proof}

Now, we prove the existence result.
\begin{theorem}
	\label{t-low-upp}
Suppose that $\gamma, \delta$ are lower and upper solutions of problem (\ref{pr*}), respectively, and $\gamma \leq \delta$. Then problem (\ref{pr*}) has at least one solution $u$ such that $$\gamma(t) \leq u(t) \leq \delta(t), \ \ \ for \ all \ t \in I.$$
\end{theorem}
\begin{proof}
First, we consider the following modified problem
\begin{eqnarray}
\begin{cases}
D^{\alpha} u(t)-\lambda u(t)+ f(t,p(t,t^{2-\alpha}u(t)))= 0,  \ \ \ \ \ t\in (0,1),\\
\lim\limits_{t \rightarrow 0^+} t^{2-\alpha}u(t)= A,   \ \ \ \                u'(1)= B,
\end{cases} \label{pr***}
\end{eqnarray}
where $p(t,x)=\max \lbrace t^{2-\alpha}\gamma(t),\ \min \lbrace x,t^{2-\alpha} \delta(t) \rbrace  \rbrace $, $t \in I$ and $x \in \R$.

Next, let us prove that Problem (\ref{pr***}) is solvable, and that all of the solutions are in $[\gamma,\delta]$. 

Suppose that $u$ is a solution of (\ref{pr***}). Then, by the definition of $\gamma$ and $\delta$ we have
$$ \gamma'(1) \leq u'(1) \leq \delta'(1) \ \ and \ \ \lim_{t\rightarrow 0^+}t^{2-\alpha} \gamma(t)\leq  \lim_{t\rightarrow 0^+}t^{2-\alpha} u(t) \leq  \lim_{t\rightarrow 0^+}t^{2-\alpha} \delta(t).$$

Assume that $u\leq \gamma$, $u \not \equiv \gamma$,  on $(0,1]$. So, using the linearity of the Riemann-Liouville derivative, we have that 
\begin{eqnarray*}
0 \leq \sigma_\gamma(t)&=&D^{\alpha} u(t) -\lambda u(t)+ f(t,t^{2-\alpha}\gamma(t))-D^{\alpha} \gamma(t) +\lambda \gamma(t)- f(t,t^{2-\alpha}\gamma(t))\\
&=&D^{\alpha}(u-\gamma)-\lambda(u-\gamma), \ \ \ 0<t<1,
\end{eqnarray*}
and
\begin{equation*}
\lim\limits_{t \rightarrow 0^+} t^{2-\alpha}(u-\gamma)(t)\geq 0, \ \ \ (u-\gamma)'(1)\geq 0.
\end{equation*}

As consequence, from Lemma \ref{l-(a,b)}, we have that $u\geq \gamma$ on $(0,1]$ and we attain a contradiction.

Thus, by denoting $v=u-\gamma$, we have that there exists $t_{0}\in (0,1)$ such that $v(t_{0})>0$.

 If there exists $t_{1}\in (0,t_{0})$ such that $v(t_{1})<0$. We have that there is $t_2 \in (t_1,t_0)$ such that $v(t_2)=0$ and $v<0$ on $(t_1,t_2)$.
 
 Now, we have two possibilities: either exists $t_3 \in (0,t_2)$ such that $v(t_3)=0$, with $v<0$ on $(t_3,t_2)$; or $v<0$ on $(0,t_2)$.

In the first case, we have that $v$ satisfies that
\begin{equation*}
D^{\alpha} v(t)-\lambda v(t) \ge 0,  \; t \in (t_3,t_2), \qquad
v(t_3)=v(t_{2})=0.
\end{equation*} 

Now, by a direct application of \cite[Corollary 3.5]{ref2} ($\lambda_1^* > \lambda_1$) we have that $v \ge 0$ on $(t_3,t_2)$, which contradicts the existence of $t_1$.

In the second situation, we deduce that

\begin{equation*}
D^{\alpha} v(t)-\lambda v(t) \ge 0,  \; t \in (0,t_2), \qquad
\lim\limits_{t\rightarrow 0^+} t^{2-\alpha} v(t)\ge 0 =v(t_{2}).
\end{equation*} 

Since $v<0$ on $(0,t_2)$, it is obvious that from previous expression, we deduce that $\lim\limits_{t\rightarrow 0^+} t^{2-\alpha} v(t)= 0$. So, the contradiction comes again from \cite[Corollary 3.5]{ref2}.

Using similar arguments we deduce that $u \leq \delta$ on $(0,1]$. \\

Therefore, we conclude that every solution $u$ of the modified problem (\ref{pr***}) is such that $\gamma \leq u \leq \delta$ on $(0,1]$.\\

Now, we shall verify that problem (\ref{pr***}) has at least one solution in $C_{2-\alpha}(I)$.

Let consider the operator $S: E \longrightarrow E$ as follows
\begin{equation*}\label{ope}
Su(t) = \int_{0}^{1} G(t,s) f(s,s^{2-\alpha}u(s))ds + A\ v_{1}(t)+ B \ v_{2}(t), \ \ \ \ \ 0<t\leq 1,
\end{equation*}
where $G$ is the Green's function given by \eqref{G}, $v_{1}$ and $v_2$ defined on \eqref{e-v1} and \eqref{e-v2} respectively.

Notice that, from Lemma \ref{l-(a,b)}, finding a fixed point of $S$ is equivalent to finding a solution of problem (\ref{pr***}).

In addition, the functions $u$, $\gamma$ and $\delta$ belong to $C_{2-\alpha}(I)$, then the truncated function $p(t,t^{2-\alpha}u(t))$ is continuous and bounded on $I$. And so, by the continuity of the function $f$, there exists a constant $C$ such that 
\begin{equation*}
C= \max_{t \in I, t^{2-\alpha}\gamma(t)\leq x \leq  t^{2-\alpha}\delta(t)}\vert f(t,x) \vert +1.
\end{equation*} 
Set 
\begin{equation*}
K=\max  \lbrace   A \Vert v_{1} \Vert_{2-\alpha} +B \Vert v_{2} \Vert_{2-\alpha}  \rbrace + \frac{ C\ M}{\alpha (\alpha-1)},
\end{equation*}
and 
$$D=\lbrace  u \in C_{2-\alpha}(I): \Vert u \Vert_{2-\alpha} \leq K  \rbrace.$$

Clearly $D$ is a closed and convex set of $C_{2-\alpha}(I)$ and that $S$ maps $D$ into $D$.

Similary, as in the proof of Lemma \ref{op1}, we conclude that $S$ satisfies the assumptions of Schauder's fixed point theorem \cite{ref7}. Which ends the proof.
\end{proof}

\begin{remark}
	Notice that from expression \eqref{e-sol-nh}, arguing as in the proof of Theorem \ref{t-low-upp}, we conclude that all the existence and uniqueness results proved in sections 4 and 5 for Problem \eqref{pr} remains valid for the non homogeneous one \eqref{e-non-homogeneous}.
\end{remark}
\section{Examples}
In this section, we give some examples to ullistrate our results.
\begin{example}
Consider the fractional differential equation (\ref{pr}) with 
$$f(t,u(t))=(1+t)\log(2+u(t)).$$
It is clear that $f$ is a nonnegative continuous function on $I\times [0,\infty)$.\\ Moreover, for $u>0$, 
$\min\limits_{t\in I}\frac{f(t,u)}{u}=\frac{\log(2+u)}{u}$ and $\max\limits_{t\in I}\frac{f(t,u)}{u}=2 \frac{\log(2+u)}{u}$.\\
We also obtain $f_{0}=\infty$ and $f^{\infty}=0$. Then, by Theorem \ref{th00} we conclude that, for $\lambda>\lambda_{1}^{*}$, problem (\ref{pr}) has a positive solution.\\

Now, let us consider the fractional differential equation (\ref{pr}) with 
$$f(t,u(t))=(2-t)u^{a}(t), \textrm{ for } a>1.$$
Clearly, assumption (H) is satisfied and for $u>0$, $\min\limits_{t\in I}\frac{f(t,u)}{u}= 2 u^{a-1}$ and $\max\limits_{t\in I}\frac{f(t,u)}{u}=u^{a-1}$.
A simple calculation yields to $f^{0}=0$ and $f_{\infty}=\infty$.\\
Therefore, by Theorem \ref{th00}, we conclude that problem (\ref{pr}) has a positive solution for $\lambda>\lambda_{1}^{*}$. 
\end{example}
\begin{example}
Consider the problem (\ref{pr}) with $\alpha=\frac{5}{3}$ and for $t \in I$ and \\$u \in [0,\infty)$, 
$$f(t,u(t))= t \frac{u(t)+1}{u(t)+2}.$$
By direct calculation, we obtain for $k=\frac{1}{4}$, that 
 \[
\vert f(t,u)-f(t,v) \vert \leq \frac{1}{4} \vert u-v \vert.
\] 
And so, assumption (H$^{*}$) is satisfied.
Thus, from Theorem \ref{thuq}, problem (\ref{pr}) has a unique solution provided that 
$\frac{K \ M}{\alpha(\alpha-1)} \simeq 0.2331\times M\ <1$.
\end{example}
\begin{example}
For any $p>1$ we define the $p$--Laplacian function as $\phi_p(x)= x\|x|^{p-1}$, $x \in \R$. For $t \in I$ and $u \in \R$, we define 
$$f(t,u)= \phi_p(u)-\lambda t^{\alpha}.$$
Consider Problem \eqref{pr*}, with $ \lambda_{1}^{*}<\lambda \le 0$, $A=0$ and $-\alpha \le B \le 0$.

It is obvious that $\delta(t)=0$ is an upper solution of this problem.

Consider $\gamma(t)=-t^{\alpha}$, It is clear that $\gamma \leq \delta$ and that it satisfies the regularity assumptions required in Theorem \ref{t-low-upp}.
 Moreover
$$ D^{\alpha} \gamma(t)-\lambda \gamma(t)+ f(t,t^{2-\alpha}\gamma(t))= -(\Gamma(\alpha+1)+\phi_p(t^2))\leq 0$$
and
$$\lim\limits_{t \rightarrow 0^+} t^{2-\alpha}\gamma(t)=0, \ \ \ \gamma'(1)=-\alpha$$

Therefore, by Theorem \ref{t-low-upp}, the considered problem has at least one solution $u$ such that $$-t^\alpha \leq u(t) \leq 0, \ \ \ for \ all \ t \in I.$$

\end{example}

\end{document}